\documentclass[letterpaper,10 pt,conference]{ieeeconf}  
\IEEEoverridecommandlockouts                              

\usepackage{comment}

\usepackage{subfigure}

\usepackage{xcolor}
\usepackage{xparse}
\usepackage{mathrsfs}
\usepackage{graphics} 
\usepackage{amsmath} 
\usepackage{amssymb}  
\usepackage{amsthm}
\usepackage{enumitem} 
\usepackage{hyperref}%
\hypersetup{colorlinks=true,  linkcolor=blue,  breaklinks=true,  urlcolor=blue,  citecolor=blue}

\usepackage{soul}
\usepackage{url}

\usepackage[prependcaption, colorinlistoftodos]{todonotes}

\DeclareMathOperator*{\rank}{rank}

\newcommand\oprocendsymbol{\hbox{$\triangle$}}
\newcommand\oprocend{\relax\ifmmode\else\unskip\hfill\fi\oprocendsymbol}

\DeclareSymbolFont{bbold}{U}{bbold}{m}{n}
\DeclareSymbolFontAlphabet{\mathbbold}{bbold}

\newtheorem{theorem}{Theorem}
\newtheorem{remark}[theorem]{Remark}
\newtheorem{example}[theorem]{Example}
\newtheorem{lemma}[theorem]{Lemma}
\newtheorem{definition}[theorem]{Definition}
\newtheorem{proposition}[theorem]{Proposition}

\newtheorem*{problem*}{Problem}

\newcommand {\be}{\begin{equation}}
\newcommand {\ee}{\end{equation}}

\newcommand{\im}{{\rm Im}}

\title{\LARGE\bf  Data-Driven Reduced-Order Unknown-Input Observers}

\author{Giorgia Disar\`o and Maria Elena Valcher
\thanks{G. Disar\`o and  M.E. Valcher are with the Dipartimento di Ingegneria dell'Informazione,
 Universit\`a di Padova,
    via Gradenigo 6B, 35131 Padova, Italy, e-mail:  \texttt{giorgia.disaro@phd.unipd.it, meme@dei.unipd.it}}
   }
  \date{}

\begin{document}
\maketitle
\begin{abstract}
In this paper we propose a data-driven approach to the design of  reduced-order unknown-input observers (rUIOs).
We first   recall the model-based solution, by assuming  a problem set-up slightly different from those traditionally adopted in the literature, in order to be able to easily adapt it to the data-driven scenario. 
Necessary and sufficient conditions for the existence of a reduced-order  unknown-input observer, whose matrices can be derived from a sufficiently rich set of collected historical data, are first derived and then proved   to be equivalent to  the ones obtained in the model-based framework. Finally, a numerical example is presented, to validate the effectiveness of the proposed scheme. \end{abstract}

\section{Introduction} \label{intro}
The problem of estimating the state of a dynamical system is of paramount importance in a large number of control applications, first of all state feedback stabilization. Indeed, the state of a system is very often not accessible and hence one needs to estimate it from input and output measurements. However, in a lot of practical situations  one cannot even assume 
 to perfectly know all the inputs that are acting on the system, because the process under consideration is subject to disturbances or actuator faults. Therefore, since the seventies, the control community has  put lots of efforts in finding solutions to the state estimation problem in the presence of unknown inputs acting on the system. Several methods have been employed, ranging from algebraic \cite{Hou-Muller1, Kudva, doppioWang} and geometric \cite{Bhatta} methods to generalized inverse approaches \cite{Kurek,Miller}. 
The majority of the existing solutions requires the perfect knowledge of the system, namely that the matrices involved in the process description are available. However, in practical situations this is not always the case and very often one has to deal with black box models, relying only on the information provided by the inputs and the outputs of the system. On the other hand, nowadays, in the big data era, large amounts of data can be collected and used to get insights into the process that has generated them. 
Hence data-driven techniques in the field of control theory have gained increasing attention. 
Data-driven methods have been proposed, in particular,
  to tackle the state estimation problem \cite{Kaynak_Automatica2023,DingIFAC2011,Mishra22}, and more specifically the unknown-input state estimation problem. In particular, in \cite{Ferrari-Trecate}, a novel data-driven unknown-input observer (UIO), based on behavioral system theory and the result known as Fundamental Lemma  proposed by Jan Willems and coworkers \cite{WillemsPE}, has been proposed. Necessary and sufficient conditions on the data collected from the system for  the existence of a UIO that makes the state estimation error converge asymptotically to zero, regardless of the unknown inputs, have been derived. 
  In \cite{TAC-UIO} the design of full-order UIOs has been further explored, by providing weaker conditions for problem solvability, and a complete parametrization of the UIOs one can derive from a given set of historical data. Moreover, it has been shown that the data-driven approach provides a problem solution under the same conditions under which a UIO can be derived 
  from the complete knowledge of the system matrices.
  The   algorithms  proposed in \cite{TAC-UIO,Ferrari-Trecate} are purely data-driven, namely they do not require any preliminary identification step. Indeed,   there are mainly two approaches to exploit data. One approach first identifies the original model, based on the collected data, and then applies  standard model-based techniques to the estimated system. The second one, instead, does not require any preliminary identification phase, and designs the UIO directly from the data. 

When the system whose state we want to estimate is extremely complex, 
implementing a full order UIO may be particularly demanding.
Indeed, 
reduced-order observers have   been widely investigated, from a model-based perspective,  due to their parsimonious nature that is always a desirable characteristic in engineering applications. In \cite{Darouach3,Guan}, a reduced-order unknown-input state estimator  (whose dimension is equal to the difference between the dimension of the state and the dimension of the unknown input) has been proposed, by first eliminating the effect of the unknown input on part of the state variables, and then designing a conventional Luenberger observer for the subsystem driven by known inputs only. A uniform design procedure for constructing reduced-order   unknown-input observers   (rUIOs) of order either equal to the difference between  the dimension of the state and the dimension of the unknown input, or to the difference between the dimension of the state and the dimension of the output, has been proposed in \cite{Hou2}. The second type of reduced-order unknown-input observers has been investigated also in \cite{Olfa, Kudva}. However, to the best of the authors' knowledge, rUIOs  have never been  addressed from a data-driven perspective.  

In this paper 
we propose a data-driven approach to the design of reduced-order unknown-input observers, by adopting
a hybrid solution, since we first identify from data   the output matrix of the data-generating system and then we leverage solely the collected data to design the rUIO. The result is not only an algorithm for state estimation of lower complexity with respect to the full-order ones, but also a less demanding procedure to generate the observer matrices from the collected data, due to their lower dimensions.

The results proposed in this paper clearly bear  similarities  with those derived in \cite{TAC-UIO,Ferrari-Trecate}
where a  data-driven approach to the design of full-order UIOs is proposed.
However, adapting the traditional model-based methods for   rUIO design to the data-driven context is not immediate.
Indeed, classic model-based techniques have either introduced the restrictive hypothesis that the output variables are   a subset of the state variables (see, e.g., \cite{Kudva}), or have resorted to  a generic change of basis in the state space  \cite{Hou2}, which would 
be difficult to extend  to the  data-driven approach.
So, the first step has been to revise the model-based solution to the problem, in such a way that its extension to, and comparison with, the one we propose based on collected data is possible.
The necessary and sufficient conditions for the existence of a reduced-order unknown-input observer   provided,  e.g., in \cite{Hou2,Kudva}  hold also in our setting, but need to be particularized to our specific description of the system. 
Based on them,  we propose a data-driven algorithm to solve the   problem. More in detail, we provide necessary and sufficient conditions for the existence of a reduced-order data-driven UIO and we show that they are actually equivalent to the ones obtained in the model-based approach. This means that the data-driven implementation does not impose additional assumptions that would be unnecessary if we knew the system matrices.
On the other hand, the possibility to effectively design an rUIO from data, by avoiding   redundancy and minimizing the computational effort, without affecting the estimation performance, is quite important from a practical point of view. 
 
The paper is organized as follows. Section \ref{problem} introduces the rUIO design problem and presents a revised solution in the model-based framework. Section \ref{dd_ROUIO} proposes the problem solution   by using a data-driven approach.
Section \ref{sec:Example} illustrates the paper results by means of a numerical example.
Section \ref{sec:Conc} concludes the paper.

\smallskip
{\bf Notation.} Given a matrix $M\in {\mathbb R}^{p \times m}$, we denote by $M^\dag\in {\mathbb R}^{m \times p}$ its {\em Moore-Penrose inverse} \cite{BenIsraelGreville}. 
Note that if $M$ is of full row rank, then $M^\dag = M^\top(MM^\top)^{-1} $. 
The null and column space of $M$ are denoted by $\ker{(M)}$ and $\im (M)$, respectively. \\
Given a vector  sequence $v(t)\in\mathbb{R}^n$, where $t\in \mathbb Z _+$, we use the notation $\{v(t)\}_{t=0}^{N}$, $N\in\mathbb Z _+$, to indicate the sequence of vectors $v(0),\dots,v(N)$.  
\medskip

\section{Problem Formulation} \label{problem}
Consider a discrete-time linear time-invariant (LTI) system $\Sigma$, described by the following equations: 
\begin{subequations}\label{sys}
\begin{eqnarray}
x(t+1)&=&Ax(t)+Bu(t)+Ed(t), \label{sys.eq1}\\
y(t)&=&Cx(t), \label{sys.eq2}
\end{eqnarray}
\end{subequations}
where $t\in\mathbb Z_+$, $x(t)\in\mathbb{R}^n$ is the state of the system, $u(t)\in\mathbb{R}^m$ is the (known) 
 control input, $d(t)\in\mathbb{R}^q$ is the unknown input or disturbance, and $y(t)\in\mathbb{R}^p$ is the output. The dimensions of the system matrices are omitted, as they can be   deduced from the dimensions of the system variables.   The  analysis carried out in this paper would still hold, with minor changes, if we replaced the output equation in \eqref{sys.eq2} with $y(t) = C x(t) + D u(t)$. However, in order not to make the subsequent calculations unnecessarily  involved, in the following we assume that $y(t)$ only depends on $x(t)$.  Without loss of generality, we assume that $E\in \mathbb{R}^{n\times q}$ is of full column rank,  and 
that  $C\in \mathbb{R}^{p\times n}$ is of full row rank, i.e., $\rank (E) = q$ and
$\rank(C)=p$. 
If $E$ is not of full column rank, we can express it as the product of a full column rank matrix and a full row rank matrix, and define a new disturbance vector. 
On the other hand, if $C$ does not have full row rank, we can replace it 
with a maximal set of linearly independent rows. 
From a practical viewpoint,  unless one is trying to implement some fault detection strategy, in which case redundancy may be useful, 
it is reasonable to 
assume that  
measurement devices are located in such a way to maximize the collected information. 
Without loss of generality,  we assume that  $C=[\ C_1\ | \ C_2 \ ]$, with $C_1\in\mathbb{R}^{p\times (n-p)}$ and $C_2\in {\mathbb R}^{p \times p}$ nonsingular, so that 
  after partitioning the state vector as $x(t) =
[ \  x_1(t)^\top \ x_2(t)^\top ]^\top$, where $x_1(t)\in\mathbb{R}^{n-p}$ and $x_2(t)\in\mathbb{R}^p$,
the system output equation in \eqref{sys.eq2} can be rewritten as 
\be \label{new_out}
y(t) = C_1 x_1(t) + C_2 x_2(t).
\ee
By premultiplying both sides of \eqref{new_out} by $C_2^{-1}$, we obtain 
\be \label{x_2}
x_2(t) = C_2^{-1}y(t)-C_2^{-1}C_1 x_1(t).
\ee
Therefore, if we can estimate the first part of the state vector, namely $x_1(t)$, we can easily recover also the remaining state variables by making use of \eqref{x_2}. 

\begin{definition} \label{def:uio}
An LTI system $\hat \Sigma$ of order $n-p$, described by the equations
\begin{subequations}\label{uio}
\hspace{-0.2cm}
{\small
\begin{eqnarray}
z(t+1)\!\!\!&=&\!\!\!A_{UIO}z(t)+B_{UIO}^u u(t)+B_{UIO}^y y(t), \label{uio.eq1}\\
\hat x_1(t)\!\!\!&=&\!\!\!  z(t)+ D_{UIO}y(t), \label{uio.eq21}\\
\hat x_2(t)\!\!\!&=&\!\!\! - C_2^{-1} C_1  z(t)+(C_2^{-1}- C_2^{-1} C_1 D_{UIO}) y(t), \ \ \label{uio.eq2}
\end{eqnarray}}
\end{subequations}
where $t\in\mathbb Z_+$, $z(t)\in\mathbb{R}^{n-p}$ is the state and $\hat x(t) = [\hat x_1 (t)^\top \  \hat x_2 (t)^\top]^\top \in\mathbb{R}^n$ is the output,  is a {\em reduced-order unknown-input observer (rUIO)}
for    system $\Sigma$ in \eqref{sys.eq1}-\eqref{sys.eq2} if $e(t) \triangleq x(t)-\hat x(t)$ asymptotically converges to zero for every choice of  $z(0)$ and every input/output pair $(\{u(t)\}_{t\in {\mathbb Z}_+}, \{y(t)\}_{t\in {\mathbb Z}_+})$ of the system \eqref{sys.eq1}-\eqref{sys.eq2}.
\end{definition}

 In other words, a reduced-order unknown-input observer is an LTI system of dimension lower than the dimension of the system $\Sigma$ that, when fed by the input/output trajectories of $\Sigma$, generated corresponding to an arbitrary $x(0)$ and an arbitrary disturbance $d(t)$, provides as its output an asymptotic estimate 
 of the state of $\Sigma$, independently of its initial condition  $z(0)$.  
 \smallskip

Clearly, by the way we have defined it,  a reduced-order UIO $\hat \Sigma$ exists
if and only if a full-order UIO for $x_1(t)$ alone, described by  \eqref{uio.eq1} and \eqref{uio.eq21} (see \cite{Darouach2,Darouach}),  exists. The ``only if" part is obvious. Conversely, if the observer  \eqref{uio.eq1} - \eqref{uio.eq21}
ensures that $e_1(t) \triangleq x_1(t)-\hat x_1(t)$ converges to zero asymptotically, 
then by making use of \eqref{uio.eq2} we can ensure that 
\be
e_2(t) \triangleq x_2(t)-\hat x_2(t)= -C_2^{-1}C_1 e_1(t)
\label{e2}
\ee
 converges to zero in turn. 
So, from now on we will focus on the UIO 
\eqref{uio.eq1}-\eqref{uio.eq21}. To design it, it is convenient to partition  all the  matrices of the system in \eqref{sys.eq1}-\eqref{sys.eq2} conformably with the block partition of $C$, namely as
$$A = \left[\begin{array}{c|c}
A_{11} &  A_{12} \\
\hline
A_{21}&A_{22}
\end{array}\right], \ \ 
B = \left[\begin{array}{c}
B_1  \\
\hline
B_2
\end{array}\right], \ \
E = \left[\begin{array}{c}
E_1  \\
\hline
E_2
\end{array}\right].
$$
By splitting the dynamics of the two parts of the state vector, we can rewrite equation \eqref{sys.eq1} 
as  
\begin{subequations}
\begin{eqnarray}
\!\!\!\!\!\!x_1(t+1)\!\!\!\!\!&=&\!\!\!\!\!A_{11}x_1(t)+A_{12}x_2(t)+B_1u(t)+E_1d(t) \label{newsys.eq1}\\
\!\!\!\!\!\!x_2(t+1)\!\!\!\!\!&=&\!\!\!\!\!A_{21}x_1(t)+A_{22}x_2(t)+B_2u(t)+E_2d(t) \label{newsys.eq2}.
\end{eqnarray}
\end{subequations}
If we now substitute equation \eqref{x_2} 
 in \eqref{newsys.eq1}, we get \begin{eqnarray} \label{new_x1}
x_1(t+1) 
& = & [A_{11} - A_{12} C_2^{-1} C_1] x_1(t)+A_{12}C_2^{-1} y(t)\nonumber  \\
&+&B_1u(t)+E_1d(t).
\end{eqnarray}
  By making use of equations \eqref{x_2}, \eqref{newsys.eq2} and \eqref{new_x1}, and the UIO description in \eqref{uio.eq1}-\eqref{uio.eq21},  we   derive the dynamics of $e_1(t) = x_1(t)-\hat x_1 (t)$, i.e., 
  \begin{eqnarray}
 e_1(t+1) \!\!\!\!\!&=&\!\!\!\!\! x_1(t+1)-\hat x_1(t+1) \nonumber \\
\!\!\!\!\!&=&\!\!\!\!\! x_1(t+1) - z(t+1) -D_{UIO}y(t+1)  \nonumber \\
\!\!\!\!\!&=&\!\!\!\!\! x_1(t+1) - A_{UIO}z(t)-B_{UIO}^u u(t)-B_{UIO}^y y(t)  \nonumber\\
\!\!\!\!\!&-&\!\!\!\!\!D_{UIO}C_1 x_1(t+1) - D_{UIO} C_2 x_2(t+1)  \nonumber\\
\!\!\!\!\!&=&\!\!\!\!\!(I-D_{UIO}C_1)x_1(t+1)-A_{UIO}\hat x_1(t)  \nonumber\\
\!\!\!\!\!&+&\!\!\!\!\! A_{UIO}D_{UIO} y(t)-B_{UIO}^u u(t) - B_{UIO}^y y(t)  \nonumber\\
\!\!\!\!\!&-&\!\!\!\!\! D_{UIO} C_2 \big(A_{21}x_1(t)+A_{22}x_2(t)+B_2u(t) \nonumber \\
\!\!\!\!\!&+&\!\!\!\!\!E_2d(t)\big) \nonumber\\
\!\!\!\!\!&=&\!\!\!\!\! A_{UIO} e_1(t) + \big[(I-D_{UIO}C_1)(A_{11}-A_{12}C_2^{-1}C_1) \nonumber \\
\!\!\!\!\!&-&\!\!\!\!\!A_{UIO}-D_{UIO}C_2A_{21}+D_{UIO}C_2A_{22}C_2^{-1}C_1\big]x_1(t)  \nonumber\\
\!\!\!\!\!&+&\!\!\!\!\! \big[(I-D_{UIO}C_1)A_{12}C_2^{-1}+A_{UIO}D_{UIO}-B_{UIO}^y \nonumber \\
\!\!\!\!\!&-&\!\!\!\!\!D_{UIO}C_2A_{22}C_2^{-1}\big]y(t) \nonumber\\
\!\!\!\!\!&+&\!\!\!\!\! \big[(I-D_{UIO}C_1)B_1 -D_{UIO}C_2B_2-B_{UIO}^u\big]u(t) \nonumber\\
\!\!\!\!\!&+&\!\!\!\!\! \big[(I-D_{UIO}C_1)E_1-D_{UIO}C_2E_2\big]d(t). \nonumber
 \end{eqnarray} 
Therefore, $e_1(t)$ is independent of the disturbance $d(t)$ and asymptotically convergent to zero, for every choice of $u(t), t\in\mathbb Z_+$, $x(0)$ and $z(0)$, if and only if the following conditions hold: 
\begin{subequations}
\begin{align}
&A_{UIO} \ \text{is Schur stable} \label{cond1} \\
&A_{UIO} = (I-D_{UIO}C_1)(A_{11}-A_{12}C_2^{-1}C_1) \nonumber \\
&\quad \quad  \ \  -D_{UIO}C_2 (A_{21} -A_{22}C_2^{-1}C_1) \label{cond2} \\
&B_{UIO}^u = (I-D_{UIO}C_1)B_1-D_{UIO}C_2B_2 \label{cond3} \\
&B_{UIO}^y = A_{UIO}D_{UIO}+(I-D_{UIO}C_1)A_{12}C_2^{-1} \nonumber \\
&\quad \quad  \ \ -D_{UIO}C_2A_{22}C_2^{-1} \label{cond4} \\
&(I-D_{UIO}C_1)E_1 - D_{UIO}C_2E_2= 0. \label{cond5}
\end{align}
\end{subequations}
When so, the state estimation error on $x_1(t)$ obeys the autonomous asymptotically stable dynamics 
\be
e_1(t+1) = A_{UIO}e_1(t).
\label{err1}
\ee
Therefore, the system in \eqref{uio} is an rUIO if and only  if   conditions  \eqref{cond1}$
\div$\eqref{cond5} hold.  
\smallskip

We now introduce the concept of {\em acceptor}, previously adopted in the context of behavior theory \cite{BehaviorObservers}.
\medskip

\begin{definition} \label{def_accept} Given system $\Sigma$, described by the equations \eqref{sys.eq1}-\eqref{sys.eq2}, we say that an LTI system $\tilde \Sigma$ described by
\begin{eqnarray*}
z(t+1) &=& \tilde A z(t) + \tilde B \begin{bmatrix} u(t)\cr y(t)\end{bmatrix},\\
\tilde x(t) &=& \tilde C z(t) + \tilde D \begin{bmatrix} u(t)\cr y(t)\end{bmatrix},
\end{eqnarray*}
where $t\in\mathbb Z_+$, $z(t)\in\mathbb{R}^{n_z}$ is the state of the system, $u(t)\in\mathbb{R}^m$ and  $y(t)\in\mathbb{R}^p$ are the system inputs, and $\tilde x(t)\in\mathbb{R}^n$ is the output, 
is  an {\em acceptor} for $\Sigma$ if 
for every    input/output/state trajectory $(\{u(t)\}_{t\in {\mathbb Z}_+},\{y(t)\}_{t\in {\mathbb Z}_+},\{x(t)\}_{t\in {\mathbb Z}_+})$  generated by $\Sigma$, there exists an initial condition $z(0)$ for $\tilde \Sigma$ such that 
$\{\tilde x(t)\}_{t\in {\mathbb Z}_+} = \{x(t)\}_{t\in {\mathbb Z}_+}$ is the output of $\tilde \Sigma$ 
corresponding to the input pair $(\{u(t)\}_{t\in {\mathbb Z}_+},\{y(t)\}_{t\in {\mathbb Z}_+})$  and the initial condition $z(0)$.
\end{definition}

The following   result, 
that will be used later in the paper, formalizes the fact that a system described as in  \eqref{uio}  is an  acceptor  for   $\Sigma$ if and only  if  conditions  \eqref{cond2}$\div$\eqref{cond5} hold.
 \smallskip
  
 \begin{lemma} \label{acceptor}
 Given system $\Sigma$, described by the equations \eqref{sys.eq1}-\eqref{sys.eq2}, an LTI
 system  described  as in  \eqref{uio}  is   an  acceptor  for   $\Sigma$ if and only if its matrices satisfy \eqref{cond2}$\div$\eqref{cond5}.
 \end{lemma}
 \begin{proof}
If conditions \eqref{cond2}$\div$\eqref{cond5} hold, the dynamics of the estimation error is described as in \eqref{err1}.
Let $(\{u(t)\}_{t\in {\mathbb Z}_+},\{y(t)\}_{t\in {\mathbb Z}_+},\{x(t)\}_{t\in {\mathbb Z}_+})$ be an  input/output/state trajectory generated by $\Sigma$.
If we assume $z(0)= x_1(0) - D_{UIO} y(0)$ then $\hat x_1(0)=x_1(0)$ and  $e_1(0) = 0$. Therefore $e_1(t)=0$ for every $t\in {\mathbb Z}_+$. This ensures that $e_2$ is identically zero, in turn, and hence $\hat x(t)=x(t)$ for every $t\in {\mathbb Z}_+$.
\smallskip

Conversely, if at least one of the conditions \eqref{cond2}$\div$\eqref{cond5} does not hold, the dynamics of the estimation error is 
not that of an autonomous system. So, it is always possible to find an initial condition $x(0)$ and input signals 
$\{u(t)\}_{t\in {\mathbb Z}_+}$ and $\{d(t)\}_{t\in {\mathbb Z}_+}$ such that for every $z(0)$ the estimation error is not identically zero.
 \end{proof}
 
It is worth remarking that  while a UIO is always  an acceptor, the converse holds if and  if  also condition \eqref{cond1} holds. 
 \smallskip
 
 We now provide necessary and sufficient conditions for the solvability of \eqref{cond1}$\div$\eqref{cond5} and thus for the existence of a reduced-order UIO \eqref{uio}.
In \cite{Hou2} these conditions have been derived under the assumption that $C=[\ 0 \ | \ I_p \ ]$, a situation we can always reduce ourselves to by resorting to a suitable change of basis (see  \cite{Kudva}).
%
Therefore, in the following lemma we only recall these conditions without providing the proof. 

\begin{lemma}\label{cns} There exist matrices $A_{UIO}, B_{UIO}^u, B_{UIO}^y$ and $D_{UIO}$ of suitable sizes that satisfy
conditions  \eqref{cond1}$\div$\eqref{cond5},  and hence there exists an  rUIO of the form \eqref{uio}, if and only if  
 \vspace{.2cm}

\noindent (a) ${\rm rank} (CE)={\rm rank}(E)=q$, and  \vspace{.2cm}

\noindent (b) $
{\rm rank} \begin{bmatrix} zI_n - A & -E\cr
C & 0\end{bmatrix} = n+q, \ \forall z\in {\mathbb C}, |z|\ge 1,$
\vspace{0.2cm}

\noindent
or, equivalently (see Theorem 2 in   \cite{Darouach2}), the triple $(A,C,E)$ is strong* detectable. 
\end{lemma}

The conditions stated in the previous lemma are exactly the same conditions that guarantee the existence of a full-order unknown-input observer \cite{Darouach2,Darouach}. Therefore,  there exists a reduced-order UIO if and only if there exists a full-order UIO. 

We are now ready to formalize the problem we want to solve. 
\begin{problem*}
Given system $\Sigma$ described as in \eqref{sys.eq1}-\eqref{sys.eq2},   with unknown matrices, design (if possible) a data-driven reduced-order unknown-input   observer 
for $\Sigma$, described by equations \eqref{uio}, by making use only of some data collected during an offline finite time experiment. 
\end{problem*} 
 
\section{Data-Driven Reduced-Order UIO} \label{dd_ROUIO}
 As in \cite{TAC-UIO,Ferrari-Trecate}, we suppose that the system matrices are unknown and that we have performed an {\em offline} experiment during which we have collected some input/output/state trajectories in the time-interval $[0,T-1]$, with $T\in\mathbb Z_+$. It has already been highlighted in \cite{TAC-UIO,Ferrari-Trecate} that assuming to have access to the state,  during the offline experiment, is necessary, since it would not be possible to uniquely identify the state of the system, and hence to construct a  UIO, without knowing the dimension and the basis of the state-space.
The input/output/state trajectories can be represented by the following sequences of vectors, i.e., $u_d =  \{u_d(t)\}_{t=0}^{T-2}$,  $y_d = \{y_d(t)\}_{t=0}^{T-1}$ and  $x_d =  \{x_d(t)\}_{t=0}^{T-1}$. Even if we cannot measure the disturbance $d(t)$, it is however convenient to define also the sequence of historical unknown input data, namely $d_d = \{d_d(t)\}_{t=0}^{T-2}$. 

For the subsequent analysis, it is useful to group the data into the following matrices:
\begin{eqnarray*}
U_p &\triangleq &\begin{bmatrix}
u_d(0) & \dots & u_d(T-2)
\end{bmatrix} \in {\mathbb R}^{m \times (T-1)}, \\
X_p &\triangleq& \begin{bmatrix}
x_d(0) & \dots & x_d(T-2)
\end{bmatrix} \in {\mathbb R}^{n \times (T-1)}, \\
X_f &\triangleq &\begin{bmatrix}
x_d(1) & \dots & x_d(T-1)
\end{bmatrix} \in {\mathbb R}^{n \times (T-1)}, \\
Y_p &\triangleq &\begin{bmatrix}
y_d(0) & \dots & y_d(T-2)
\end{bmatrix}\in {\mathbb R}^{p \times (T-1)}, \\
Y_f &\triangleq &\begin{bmatrix}
y_d(1) & \dots & y_d(T-1)
\end{bmatrix} \in {\mathbb R}^{p \times (T-1)}, \\
D_p &\triangleq& \begin{bmatrix}
d_d(0) & \dots & d_d(T-2)
\end{bmatrix} \in {\mathbb R}^{q \times (T-1)}. 
\end{eqnarray*}
where the subscripts $p$ and $f$ stand for past and future, respectively. 

Before providing the data-driven formulation of the reduced-order UIO,  we introduce the following assumption (the same  we adopted in \cite{TAC-UIO}  for full-order UIOs, and that follows from more restrictive assumptions of persistence of excitation of the input sequences $u_d$ and $d_d$):

\smallskip
\noindent {\bf Assumption:}
The matrix $\begin{bmatrix}
U_p^\top &
D_p^\top &
X_p^\top 
\end{bmatrix}^\top$ is of full row rank, i.e., $m+q+n$.

\medskip
Since the historical data have been generated by the system $\Sigma$, they have to satisfy \eqref{sys.eq1}-\eqref{sys.eq2} and in particular it must hold 
\be \label{Yp}
Y_p = C X_p.
\ee
Under the previous Assumption, the matrix $X_p$ is of full row rank and thus admits a right inverse. Therefore, based on equation \eqref{Yp}, it is possible to uniquely identify the output matrix $C$ from data as 
$$C = Y_p X_p^\dag = Y_p X_p^\top (X_p X_p^\top)^{-1}.$$
Note that 
if we had assumed $y(t)=C x(t) + D u(t)$,
by making use of the identity 
$$Y_p = C X_p + D U_p= \begin{bmatrix} C & D \end{bmatrix} \begin{bmatrix} X_p\cr U_p\end{bmatrix},$$ 
and exploiting again the Assumption, we could have uniquely identified  both $C$ and  $D$.

Once we have recovered the matrix $C$ from the output/state data, we can also check if it has full row rank. 
If not, we can discard the measurements that are linearly dependent on the others. Again, there is no loss of generality in assuming that 
$C$ can be block-partitioned as $C=[\ C_1\ | \ C_2 \ ]$, where 
$C_2$ is nonsingular square\footnote{If this is not the case, we can resort to a permutation matrix $P$ and replace $C$ with $CP$, and $X_p, X_f, U_p, U_f$ with $P^\top X_p,$ $P^\top X_f,$ $P^\top U_p, P^\top U_f$.}. 
Now that we have the matrix $C$ along with its partition, 
we can split the generic state vector $x_d(t)$, belonging to the sequence of historical data $x_d$, into two blocks
$x_d(t) = \begin{bmatrix} x_{d,1}(t)^\top & x_{d,2}(t)^\top \end{bmatrix}^\top$, conformably with the block partition of $C$.
Consequently,  the
matrices of the state data  split into two parts, namely 
\begin{eqnarray*}
X_{p,1} &\triangleq& \begin{bmatrix}
x_{d,1}(0) & \dots & x_{d,1}(T-2)
\end{bmatrix} \in {\mathbb R}^{(n-p) \times (T-1)}, \\
X_{f,1}&\triangleq &\begin{bmatrix}
x_{d,1}(1) & \dots & x_{d,1}(T-1)
\end{bmatrix} \in {\mathbb R}^{(n-p) \times (T-1)}, \\
X_{p,2} &\triangleq& \begin{bmatrix}
x_{d,2}(0) & \dots & x_{d,2}(T-2)
\end{bmatrix} \in {\mathbb R}^{p \times (T-1)}, \\
X_{f,2}&\triangleq &\begin{bmatrix}
x_{d,2}(1) & \dots & x_{d,2}(T-1)
\end{bmatrix} \in {\mathbb R}^{p \times (T-1)},
\end{eqnarray*}
and we have the following identities 
\begin{eqnarray*}
X_{p,2} &=& C_2^{-1}Y_p-C_2^{-1}C_1 X_{p,1} \\
X_{f,2} &=& C_2^{-1}Y_f-C_2^{-1}C_1 X_{f,1}. 
\end{eqnarray*}

 When dealing with data-driven techniques, it is important that the collected data 
 are representative of the underlying system. 
The following definition aims at capturing this concept. 
\smallskip

\begin{definition}\label{compatible_traj} \cite{TAC-UIO,Ferrari-Trecate}
The set of   (input/output/state) trajectories $(\{u(t)\}_{t\in {\mathbb Z}_+},\{y(t)\}_{t\in {\mathbb Z}_+},\{x(t)\}_{t\in {\mathbb Z}_+})$  is said to be {\em compatible with the historical data} $(u_d, y_d, x_d)$ if 
\be\label{compatibility}
\begin{bmatrix}
u(t) \\
y(t) \\
x(t) \\
x(t+1) 
\end{bmatrix} \in 
\im \left(
\begin{bmatrix}
U_p \\
Y_p \\
X_p \\
X_f 
\end{bmatrix}
\right), \ \forall t\in {\mathbb Z}_+.
\ee
\end{definition}

\smallskip
Under the Assumption   we made on the data, it has  been proved in \cite{TAC-UIO} (see, also, \cite{Ferrari-Trecate}) that the trajectories generated by the system $\Sigma$ in \eqref{sys.eq1}-\eqref{sys.eq2} are all and only those compatible with the given historical data. 
\medskip

In \cite[Lemma 2]{Ferrari-Trecate}, it has also been shown that there exists an acceptor of order $n$ for $\Sigma$ described by
\begin{subequations}\label{FOuio}
{\small\begin{eqnarray}
z(t+1)\!\!\!&=&\!\!\!A_{UIO}z(t)+B_{UIO}^u u(t)+B_{UIO}^y y(t), \label{FOuio.eq1}\\
\hat x(t)\!\!\!&=&\!\!\!  z(t)+ D_{UIO}y(t), \label{FOuio.eq21}
\end{eqnarray}}
\end{subequations}
if and only if 
\be\label{ker_inclusion}
\ker{(X_f)}\supseteq \ker{\left(\begin{bmatrix}
U_p \\
Y_p \\
Y_f \\
X_p 
\end{bmatrix}\right)}.
\ee
This result, applied to the reduced-order scenario, leads to the following proposition (whose proof can be obtained by suitably adjusting that of Lemma 9 in \cite{TAC-UIO}). 

\begin{proposition} \label{reduced_acceptor}
There exists an acceptor described as in  \eqref{uio} for $\Sigma$ (or equivalently,  an acceptor described by \eqref{uio.eq1}-\eqref{uio.eq21} for the trajectories $(\{x_1(t)\}_{t\in {\mathbb Z}_+}, \{u(t)\}_{t\in {\mathbb Z}_+},\{y(t)\}_{t\in {\mathbb Z}_+})$ of $\Sigma$),  
whose  matrices $A_{UIO}$, $B_{UIO}^u$, $B_{UIO}^y$ and $D_{UIO}$ are built using the collected data $U_p,Y_p,Y_f,X_{p,1}$, 
if and only if 
\be\label{ker_reduced}
\ker{(X_{f,1})}\supseteq \ker{\left(\begin{bmatrix}
U_p \\
Y_p \\
Y_f \\
X_{p,1} 
\end{bmatrix}\right)}.
\ee
If so,   for every 
$ 
 \left[\begin{array}{c|c|c|c}
 S_1 & S_2 & S_3 & S_4
 \end{array}
 \right]\in {\mathbb R}^{n \times (m+2p+n-p)}$ satisfying
 \be\label{Xf1}
 X_{f,1} =  \left[\begin{array}{c|c|c|c}
 S_1 & S_2 & S_3 & S_4
 \end{array}
 \right]
\begin{bmatrix}
U_p \\
Y_p \\
Y_f \\
X_{p,1} 
\end{bmatrix},
\ee 
 the  matrices of  an acceptor can be expressed in terms of the matrices $S_1, S_2, S_3$ and $S_4$ as
\begin{eqnarray}
A_{UIO}&\triangleq& S_4 , \label{Auio} \\ 
B_{UIO}^u &\triangleq& S_1, \label{Buio_u} \\
B_{UIO}^y  &\triangleq& S_2+S_4S_3, \label{Buio_y} \\
D_{UIO} &\triangleq& S_3. \label{Duio}
\end{eqnarray}
Conversely, for every acceptor described by the  matrices $A_{UIO}$, $B_{UIO}^u$, $B_{UIO}^y$ and $D_{UIO}$,
we can obtain a solution of \eqref{Xf1} by assuming
\begin{eqnarray}
S_1&\triangleq& B_{UIO}^u , \label{S1} \\
S_2   &\triangleq& B_{UIO}^y - A_{UIO}D_{UIO}, \label{S2} \\
S_3 &\triangleq& D_{UIO}, \label{S3}\\
S_4&\triangleq& A_{UIO}. \label{S4} 
\end{eqnarray}

\end{proposition}

\begin{remark}\label{equivalence}  
From Lemma \ref{acceptor} and Proposition \ref{reduced_acceptor},
it follows that the kernels inclusion in \eqref{ker_reduced} corresponds exactly to   conditions \eqref{cond2}$\div$\eqref{cond5} derived in the model-based approach, by imposing the decoupling from all the exogenous variables in the estimation error dynamics.
Indeed, from  Proposition \ref{reduced_acceptor} it follows that  \eqref{ker_reduced} is equivalent to the existence of 
 matrices $A_{UIO}$, $B_{UIO}^u$, $B_{UIO}^y$ and $D_{UIO}$ such that
{\small 
$$
 X_{f,1} = \!\!  \left[\begin{array}{c|c|c|c}
 B_{UIO}^u &\!\!  B_{UIO}^y-A_{UIO}D_{UIO}  &\!\!  D_{UIO} &\!\!  A_{UIO}
 \end{array}
 \right]\!\!\!
\begin{bmatrix}
U_p \\
Y_p \\
Y_f \\
X_{p,1} 
\end{bmatrix}.
$$ }
If we now exploit the fact that the data have been generated by the system $\Sigma$, we can substitute $Y_f$ 
in the previous equation with the following expression
$Y_f = C_1 X_{f,1}+C_2 X_{f,2} = C_1 (A_{11}X_{p,1}+A_{12} X_{p,2}+B_1U_p+E_1D_p)+C_2(A_{21}X_{p,1}+A_{22} X_{p,2}+B_2U_p+E_2D_p) = [C_1(A_{11}-A_{12}C_2^{-1}C_1)+C_2(A_{21}-A_{22}C_2^{-1}C_1)]X_{p,1}+CB U_p+(C_1A_{12}C_2^{-1}+C_2A_{22}C_2^{-1})Y_p+CE D_p$ and obtain
\begin{eqnarray} \label{xf1}
\!\! X_{f,1} \!\!\!&=&\!\!\!
 \left[ B_{UIO}^u + D_{UIO}CB  \ | \ D_{UIO} CE \ | \ B_{UIO}^y \right.  \nonumber \\
\!\!\!&&\!\!\! \left. -A_{UIO}D_{UIO} +D_{UIO}(C_1A_{12} +C_2 A_{22})C_2^{-1} | \right.   \nonumber\\
\!\!\!&&\!\!\! \left. A_{UIO} +D_{UIO}[C_1(A_{11}-A_{12}C_2^{-1}C_1)  \right.  \nonumber \\
\!\!\!&&\!\!\! \left.+C_2(A_{21}-A_{22}C_2^{-1}C_1)] \right]
\begin{bmatrix}
U_p \\
D_p \\
Y_p  \\
X_{p,1}
\end{bmatrix}. \!\!\!\!
\end{eqnarray}

 At the same time, the historical data have to satisfy the equations of system $\Sigma$ (see \eqref{new_x1}) and thus $X_{f,1}$ can also be written as 
\begin{eqnarray} \label{xf1_sys}
\!\!\!\!\!\!\! X_{f,1} \!\!\!\!&=&\!\!\!\!
 \left[ B_1  |  E_1  |  A_{12}C_2^{-1}  | A_{11}-A_{12}C_2^{-1}C_1\right] \!\! \begin{bmatrix}
U_p \\
D_p \\
Y_p  \\
X_{p,1}
\end{bmatrix}\!. \ \ 
\end{eqnarray}
Since the matrix $\begin{bmatrix}
U_p^\top &
D_p^\top &
Y_p^\top  &
X_{p,1}^\top
\end{bmatrix}^\top$ is of full row rank \footnote{The full row rank property of the matrix $\begin{bmatrix}
U_p^\top &
D_p^\top &
Y_p^\top  &
X_{p,1}^\top
\end{bmatrix}^\top$ follows directly from the Assumption on the data and the full row rank property of the matrix $C$. Indeed, the matrix $\begin{bmatrix}
U_p^\top &
D_p^\top &
Y_p^\top  &
X_{p,1}^\top
\end{bmatrix}^\top$ can be written as 
$$
\begin{bmatrix}
U_p \\
D_p \\
Y_p \\
X_{p,1}
\end{bmatrix}=
\left[\begin{array}{c|c|c}
I &  0 & 0 \\
\hline
0 & I & 0\\
\hline
0 & 0 & \begin{array}{c|c}
C_1 & C_2 \\
\hline
I & 0 
\end{array}
\end{array}\right]
\begin{bmatrix}
U_p \\
D_p \\
X_{p,1} \\
X_{p,2}   
\end{bmatrix},$$
and hence is of full row rank since it is the product of a nonsingular square matrix and a full row rank matrix, respectively. }, 
by equating the right hand side of \eqref{xf1} and \eqref{xf1_sys}
 we obtain 
exactly the conditions in \eqref{cond2}$\div$\eqref{cond5}.
\end{remark}

Next, we show that the two conditions in \eqref{ker_inclusion} and \eqref{ker_reduced} are actually equivalent, namely 
we can build a reduced-order acceptor for the input/output/state trajectories of system $\Sigma$
if and only if we can build a full order acceptor for the same system. 

\begin{proposition}\label{ker_ridotto}
Under the Assumption on the data and the hypothesis that $C_2$ is nonsingular, conditions
\eqref{ker_inclusion} and \eqref{ker_reduced} are equivalent.
%
\end{proposition}

\begin{proof} 
Condition \eqref{ker_inclusion} implies that 
$\exists 
 \left[\begin{array}{c|c|c|c}
 T_1 & T_2 & T_3 & T_4
 \end{array}
 \right]\in {\mathbb R}^{n \times (m+2p+n)}$ s.t. 
 \be\label{Xf}
 X_f =  \left[\begin{array}{c|c|c|c}
 T_1 & T_2 & T_3 & T_4
 \end{array}
 \right]
\begin{bmatrix}
U_p \\
Y_p \\
Y_f \\
X_p 
\end{bmatrix}.
\ee 
We partition the matrix $T_4$ conformably with the partition of the vector $x$, namely 
$T_4 = \left[\begin{array}{c|c}
 T_{41} & T_{42} 
 \end{array}
 \right]$, with $T_{41}\in {\mathbb R}^{n\times (n-p)}$ and $T_{42}\in {\mathbb R}^{n\times p}$,
and we observe that 
$ X_{p,2} = C_2^{-1}Y_p-C_2^{-1}C_1 X_{p,1} $. 
This implies that 
\begin{eqnarray*}
T_4 X_p &=& \left[\begin{array}{c|c}
 T_{41} & T_{42} 
 \end{array}
 \right]\begin{bmatrix}
 X_{p,1} \\
 C_2^{-1}Y_p-C_2^{-1}C_1 X_{p,1}
  \end{bmatrix}\\
 &=& (T_{41}-T_{42}C_2^{-1}C_1) X_{p,1} + T_{42}C_2^{-1}Y_p,
\end{eqnarray*}
yielding 
\begin{eqnarray*}
X_{f,1} \!\!\!&=&\!\!\! \begin{bmatrix}  I_{n-p} & 0\end{bmatrix}  
 \big[T_1 \ | \ T_2 + T_{42}C_2^{-1} \ | \ T_3 | \\ 
&&T_{41}-T_{42}C_2^{-1}C_1\big]
 \begin{bmatrix}
U_p \\
Y_p \\
Y_f \\
X_{p,1} 
\end{bmatrix},
\end{eqnarray*}
and hence \eqref{ker_reduced} holds.

Conversely, condition \eqref{ker_reduced} implies that 
$\exists 
 \left[\begin{array}{c|c|c|c}
 S_1 & S_2 & S_3 & S_4
 \end{array}
 \right]\in {\mathbb R}^{(n-p) \times (m+2p+n-p)}$ s.t. 
$$
 X_{f,1} =  \left[\begin{array}{c|c|c|c}
 S_1 & S_2 & S_3 & S_4
 \end{array}
 \right]
\begin{bmatrix}
U_p \\
Y_p \\
Y_f \\
X_{p,1} 
\end{bmatrix}.
$$
Moreover, we have 
\begin{eqnarray*}
X_{f,2} &=& C_2^{-1}Y_f-C_2^{-1}C_1 X_{f,1} \\
&=& C_2^{-1}Y_f-C_2^{-1}C_1 \left[\begin{array}{c|c|c|c}
 S_1 & S_2 & S_3 & S_4
 \end{array}
 \right]
\begin{bmatrix}
U_p \\
Y_p \\
Y_f \\
X_{p,1} 
\end{bmatrix}\\
&=& 
\left[ -C_2^{-1}C_1S_1 \ | \ -C_2^{-1}C_1S_2 \ | \right.  \\
 && \left. C_2^{-1}-C_2^{-1}C_1S_3 \ | \ -C_2^{-1}C_1S_4
 \right]
\begin{bmatrix}
U_p \\
Y_p \\
Y_f \\
X_{p,1} 
\end{bmatrix},
\end{eqnarray*}
which leads to 
\begin{eqnarray*}
X_f \!\!\!&=&\!\!\! \begin{bmatrix}
X_{f,1}\\
X_{f,2}
\end{bmatrix} = {\small
\left[\begin{matrix}
S_1 & S_2  \\
-C_2^{-1}C_1S_1 & -C_2^{-1}C_1S_2 
\end{matrix}\right. }\\
&& \quad \quad \quad  \quad 
{\small\left.\!\!\begin{matrix}
S_3 & S4 & 0 \\
C_2^{-1}-C_2^{-1}C_1S_3 & -C_2^{-1}C_1S_4 & 0
\end{matrix} \right]
\!\!\! \begin{bmatrix}
U_p \\
Y_p \\
Y_f \\
X_{p,1} \\
X_{p,2}
\end{bmatrix}}
\end{eqnarray*}
which implies  \eqref{ker_inclusion}.
\end{proof}

So far, we have designed only a data-driven acceptor of the form \eqref{uio} for the system in \eqref{sys.eq1}-\eqref{sys.eq2}. To make this acceptor an rUIO, we need to impose a further requirement, namely we have to guarantee that the dynamics of the estimation error is not only autonomous, but also Schur stable.

\begin{theorem} \label{final}
Given the historical data $(u_d, y_d, x_d)$, satisfying the Assumption, 
there exists a reduced-order unknown-input observer   for system $\Sigma$ of the form \eqref{uio},
designed from the historical data,  if and only if $\exists \left[\begin{array}{c|c|c|c}
 S_1 & S_2 & S_3 & S_4
 \end{array}
 \right]\in {\mathbb R}^{(n-p) \times (m+2p+n-p)}$, with $S_4$ Schur stable, such that 
 \eqref{Xf1} holds.
\end{theorem}

\begin{proof} As discussed in the previous section, a system described as in  \eqref{uio}
is a reduced-order unknown-input observer   for system $\Sigma$ if and only if it is an acceptor and the 
dynamics of $e_1(t)= x_1(t)-\hat x_1(t)$ is autonomous and asymptotically stable. 
By Proposition \ref{reduced_acceptor} and the subsequent Remark \ref{equivalence}, we know that 
there  exists an acceptor   for system $\Sigma$ of the form \eqref{uio},  designed from the historical data,  if and only if $\exists \left[\begin{array}{c|c|c|c}
 S_1 & S_2 & S_3 & S_4
 \end{array}
 \right]\in {\mathbb R}^{(n-p) \times (m+2p+n-p)}$, such that 
 \eqref{Xf1} holds.
Since we have shown that 
the estimation error on the first components of the state follows the autonomous dynamics $e_1(t+1) = A_{UIO}e_1(t)=S_4e_1(t)$, such autonomous dynamics is asymptotically stable if and only if  $S_4 = A_{UIO}$ is Schur stable. 
\end{proof}

Note that the general solution to  \eqref{Xf1} is given by 
{\small
\begin{eqnarray*}
 \left[\begin{array}{c|c|c|c}
 S_1 & S_2 & S_3 & S_4
 \end{array}
 \right] &=& X_{f,1}\begin{bmatrix}
U_p \\
Y_p \\
Y_f \\
X_{p,1} 
\end{bmatrix}^\dag \\
&+&
W
\left(I-\begin{bmatrix}
U_p \\
Y_p \\
Y_f \\
X_{p,1} 
\end{bmatrix}
\begin{bmatrix}
U_p \\
Y_p \\
Y_f \\
X_{p,1} 
\end{bmatrix}^\dag\right),
\end{eqnarray*}}
where $W$ is an arbitrary matrix of suitable dimensions.  So, once condition  \eqref{ker_reduced} holds,
one has to explore if there is a solution to \eqref{Xf1}, in the set of solutions parametrized  above,  with $S_4$ Schur stable.

\section{Example}\label{sec:Example}
In this section we provide a  numerical example to validate the obtained results. 

\begin{example} \label{ex1}
Consider a system $\Sigma$  of order $n=5$ described as in \eqref{sys} for the following choice of matrices: 
\begin{align*}
A  &= \left[\begin{array}{c|c}
A_{11} &  A_{12} \\
\hline
A_{21}&A_{22}
\end{array}\right] = \left[\begin{array}{cc|ccc}
0 &  0 & 0 & 0 & 1/2 \\
1 & 0 & 0 & 0 & 3/4 \\
\hline
0 & 1& 0 & 0 & -2 \\
0 & 0 & 1 & 0 & -5/4 \\
0 & 0 & 0 & 1 & 3
\end{array}\right], \\
B &= \left[\begin{array}{c}
B_1 \\
\hline
B_2
\end{array}\right]=\left[\begin{array}{cc}
0 & 1 \\
2 & 1 \\
\hline
-2 & 1 \\
0 & 0 \\
1 & 0
\end{array}\right], \  E = \left[\begin{array}{c}
E_1 \\
\hline
E_2
\end{array}\right]=\left[\begin{array}{cc}
0 & 1 \\
0 & 0 \\
\hline
0  & 0\\
2 & 1\\
1 & 0
\end{array}\right], \\ 
C &= \left[\begin{array}{c|c}
C_1 &  C_2 
\end{array}\right] = \left[\begin{array}{cc|ccc}
0 &  1 & -1 & 2 & -1\\
0 & 0 & 2 & 0 & -1 \\
3 & 0 & 2 & -1 & 1
\end{array}\right].
\end{align*}
Historical  (both known and unknown) input data have been randomly generated, uniformly in the interval $(-5,5)$ for the known input $u(t)$, and in the interval $(-2,2)$ for the disturbance $d(t)$. The time-interval of the offline experiment has been set to $T=11$. We have collected the data corresponding to the input/output/state trajectories and then   checked that all the assumptions are satisfied and that the kernels inclusion holds. Clearly, from $Y_p$ and $X_p$ we have recovered the exact expression of $C$.
We have then set as matrices of the rUIO in \eqref{uio} the ones corresponding to the following particular solution of equation \eqref{Xf1}:
{\small
$$ \label{gen_sol}
 \left[\begin{array}{c|c|c|c}
\!\!\!B_{UIO}^u \!\!&\!\! B_{UIO}^y-A_{UIO}D_{UIO} \!\!&\!\! D_{UIO} \!\!&\!\! A_{UIO}\!\!\!
 \end{array}
 \right] = X_{f,1}\begin{bmatrix}
U_p \\
Y_p \\
Y_f \\
X_{p,1} 
\end{bmatrix}^\dag \!\!\!,$$}
namely
\begin{align*}
A_{UIO} &= \begin{bmatrix}
0.1580 & -0.4135 \\
0.3763	& 0.0029
\end{bmatrix}, \\ 
B_{UIO}^u &= \begin{bmatrix}
0.6797 & -0.8599 \\
1.8089 & 1.0409
\end{bmatrix}, \\
B_{UIO}^y &= \begin{bmatrix}
-0.1618 & 0.0889 & -0.0382 \\
0.1104 & -0.1670 & 0.3555
\end{bmatrix}, \\
D_{UIO} &= \begin{bmatrix}
0.1200 & -0.0201 & 0.3800 \\
-0.0136 & -0.0546 & 0.0136 
\end{bmatrix}. 
\end{align*}
It is easy to verify that the matrix $A_{UIO}$ is Schur stable. Finally, we have  tested the performance of the designed rUIO corresponding to the (known) input   $u(t) = \begin{bmatrix} 0.8\cos(0.2t+2) &
3t\end{bmatrix}^\top, \ t\in\mathbb{Z}_+$, and a random disturbance $d(t)$ whose first and second components take values uniformly in the interval $(-5,5)$ and $(-2,2)$, respectively.   Figures \ref{fig1}  illustrates  the state estimation error, that asymptotically converges to zero, as expected.   

\begin{figure} 
\centering
\includegraphics[width = 0.5\textwidth]{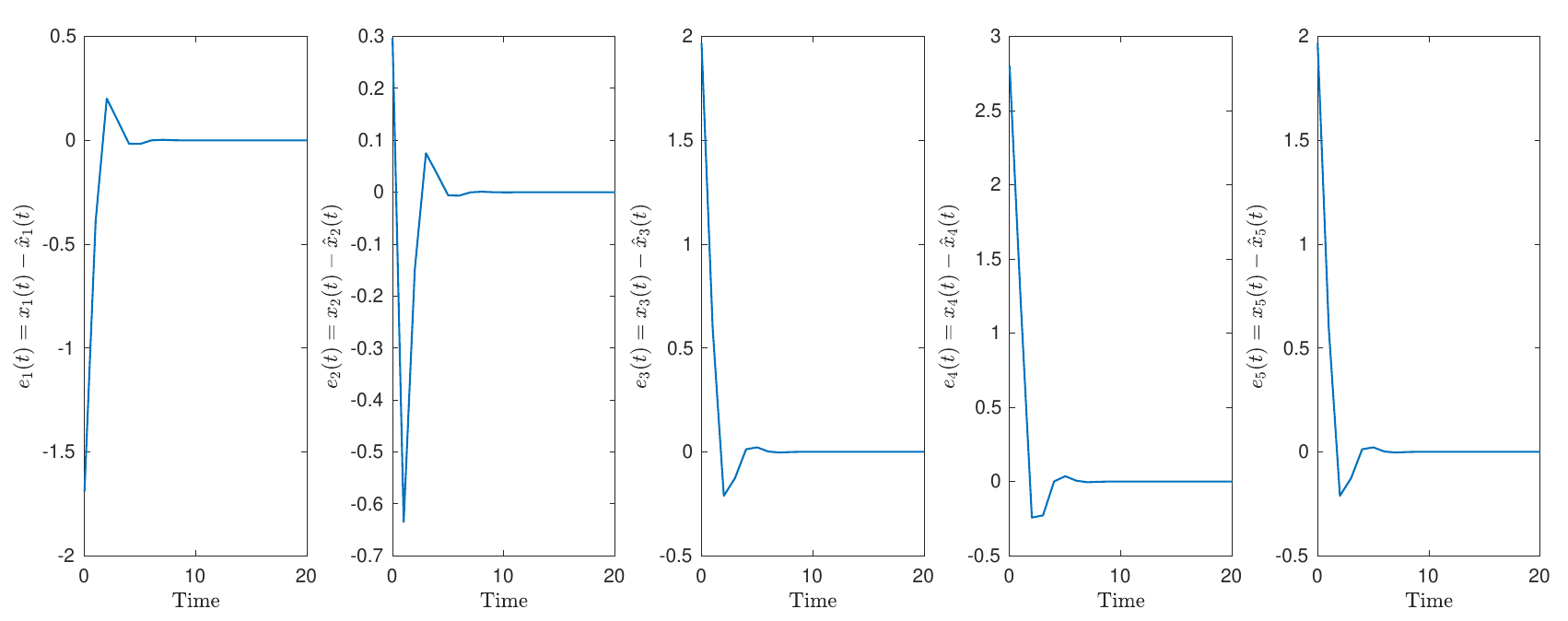}
\caption{Dynamics of the state estimation error} 
\label{fig1}
\end{figure} 

\end{example}

\section{Conclusions and future work}\label{sec:Conc}

In this paper we have proposed a data-driven technique to design a reduced-order UIO based on some collected data.
To achieve this goal we have revised the model-based approach to the problem solution, by adopting a set-up that does not introduce overly-simplified assumptions on the output matrix $C$ and at the same time does not require general changes of basis, that would be difficult to extend to the data-based approach.
This set-up, on the contrary, has allowed us to leverage the results recently obtained in 
\cite{TAC-UIO,Ferrari-Trecate} for full-order observers, and hence to obtain necessary and sufficient conditions for the problem solution. The final result is a more efficient algorithm for the state estimation, that can be obtained exactly under the same assumptions under which a full-order UIO could be designed.

The system model we have adopted assumes that the output depends on the state only.
The more general case where  $y(t)=Cx(t) + Du(t)$ represents a quite easy extension, that however would increase the complexity of the formulas, and hence we have preferred to avoid it.
The case where disturbance affects the output measurements, on the other hand, is a non-trivial extension and the subject of future investigations.

The concluding example, that we chose of small size to be able to provide its describing matrices within the space limits, clearly illustrates the excellent performance of the rUIO.

\bibliographystyle{plain}

\bibliography{Refer172}

\end{document}